\documentclass[12pt]{amsart}
\usepackage{amscd,amsmath,amsthm,amssymb}
\usepackage[left]{lineno}
\usepackage{color}
\usepackage{stmaryrd}
\usepackage[utf8]{inputenc}
\usepackage{cleveref}
\usepackage{epstopdf}
\usepackage{graphicx}
\usepackage{xcolor}

\usepackage{tikz}

\definecolor{verylight}{gray}{0.97}
\definecolor{light}{gray}{0.9}
\definecolor{medium}{gray}{0.85}
\definecolor{dark}{gray}{0.6}
\newcommand\calP{\mathcal{P}}
 \def\NZQ{\mathbb}               

 \def\KK{{\NZQ K}}
 
\def\KK{{\NZQ K}}

 \def\G{{\mathcal G}}

 \def\0b{{\mathbf 0}}

\def\reg{{\mathbf reg}}
\def\height{\operatorname{ht}}
\def\depth{\operatorname{depth}}
 \def\opn#1#2{\def#1{\operatorname{#2}}} 
 \opn\chara{char} \opn\length{\ell} \opn\pd{pd} \opn\rk{rk}
 \opn\projdim{proj\,dim} \opn\injdim{inj\,dim} \opn\rank{rank}
 \opn\depth{depth} \opn\grade{grade} \opn\height{height}
 \opn\embdim{emb\,dim} \opn\codim{codim}
 
 \opn\Tr{Tr} \opn\bigrank{big\,rank}
 \opn\superheight{superheight}\opn\lcm{lcm}
 \opn\trdeg{tr\,deg}
 \opn\reg{reg} \opn\lreg{lreg} \opn\ini{in} \opn\lpd{lpd}
 \opn\size{size} \opn\sdepth{sdepth}
 \opn\link{link}\opn\fdepth{fdepth}\opn\lex{lex}
 \opn\tr{tr}
 \opn\type{type}
 \opn\gap{gap}
 \opn\arithdeg{arith-deg}
 \opn\HS{HS}
 \opn\GL{GL}
 \opn\div{div} \opn\Div{Div} \opn\cl{cl} \opn\Cl{Cl}
 \opn\Spec{Spec} \opn\Supp{Supp} \opn\supp{supp} \opn\Sing{Sing}
 \opn\Ass{Ass} \opn\Min{Min}\opn\Mon{Mon}
 \opn\Ann{Ann} \opn\Rad{Rad} \opn\Soc{Soc}\opn\Deg{Deg}
 \opn\Im{Im} \opn\Ker{Ker} \opn\Coker{Coker} \opn\Am{Am}
 \opn\Hom{Hom} \opn\Tor{Tor} \opn\Ext{Ext} \opn\End{End}
 \opn\Aut{Aut} \opn\id{id}
 
 \opn\nat{nat}
 \opn\pff{pf}
 \opn\Pf{Pf} \opn\GL{GL} \opn\SL{SL} \opn\mod{mod} \opn\ord{ord}
 \opn\Gin{Gin} \opn\Hilb{Hilb}\opn\sort{sort}
 \opn\PF{PF}\opn\Ap{Ap}
 \opn\mult{mult}
 \opn\bight{bight}
 \opn\aff{aff}
 \opn\relint{relint} \opn\st{st}
 \opn\lk{lk} \opn\cn{cn} \opn\core{core} \opn\vol{vol}  \opn\inp{inp} \opn\nilpot{nilpot}
 \opn\link{link} \opn\star{star}\opn\lex{lex}\opn\set{set}
 \opn\width{wd}
 \opn\Fr{F}
 \opn\QF{QF}
 \opn\G{G}
 \opn\type{type}\opn\res{res}
 \opn\conv{conv}
 \opn\Ind{Ind}
 \opn\gr{gr}
 
 \def\pot#1#2{#1[\kern-0.28ex[#2]\kern-0.28ex]}

 \opn\dirlim{\underrightarrow{\lim}}
 \opn\inivlim{\underleftarrow{\lim}}
 \def\Implies{\ifmmode\Longrightarrow \else
         \unskip${}\Longrightarrow{}$\ignorespaces\fi}
 \def\implies{\ifmmode\Rightarrow \else
         \unskip${}\Rightarrow{}$\ignorespaces\fi}
 \def\iff{\ifmmode\Longleftrightarrow \else
         \unskip${}\Longleftrightarrow{}$\ignorespaces\fi}

 \let\:=\colon
 \newtheorem{Theorem}{Theorem}[section]
 \newtheorem{Lemma}[Theorem]{Lemma}
 \newtheorem{Corollary}[Theorem]{Corollary}
 
 \newtheorem{Remark}[Theorem]{Remark}

 \newtheorem{Definition}[Theorem]{Definition}

\newtheorem{Setting}[Theorem]{Setting}

 \let\epsilon\varepsilon
 \let\kappa=\varkappa
 \def\qed{\ifhmode\textqed\fi
       \ifmmode\ifinner\quad\qedsymbol\else\dispqed\fi\fi}
 \def\textqed{\unskip\nobreak\penalty50
        \hskip2em\hbox{}\nobreak\hfil\qedsymbol
        \parfillskip=0pt \finalhyphendemerits=0}
 \def\dispqed{\rlap{\qquad\qedsymbol}}
 
 \opn\dis{dis}
 \def\pnt{{\raise0.5mm\hbox{\large\bf.}}}
 
 \opn\Lex{Lex}

\begin{document}
	\title {Regularity of powers of  edge ideals of  edge-weighted integrally closed cycles}
	
	\author {Guangjun Zhu$^{\ast}$, Yijun Cui, Jiaxin Li and Yi Yang}

	\address{School of Mathematical Sciences, Soochow University, Suzhou, Jiangsu, 215006, P. R. China}

	\email{ zhuguangjun@suda.edu.cn(Corresponding author:Guangjun Zhu),\linebreak[4]
		237546805@qq.com(Yijun Cui), lijiaxinworking@163.com(Jiaxin Li), 3201088194@qq.com(Yi Yang).}
	
	\thanks{$^{\ast}$ Corresponding author}



	

	
	\maketitle
\begin{abstract}
 This paper gives exact formulas for the regularity  of powers of edge ideals of an edge-weighted  integrally closed cycle.
\end{abstract}
	\setcounter{tocdepth}{1}

    \section{Introduction}
 Let $G$ be a graph with the vertex set $V(G)$ and  the edge set $E(G)$.
 We write $xy$ for  an  edge of $G$ with $x$ and $y$ as endpoints.
Suppose ${\bf w}: E(G)\rightarrow \mathbb{Z}_{>0}$ is an edge weight function on $G$. We write $G_{\bf w}$ for the pair $(G,{\bf w})$ and call it an {\em edge-weighted} graph  (or simply weighted graph) with the underlying graph  $G$. For a weighted graph $G_{\bf w}$ with the vertex set $V(G)=\{x_1,\ldots,x_n\}$, its {\em edge-weighted ideal}  (or simply edge ideal), was introduced in \cite{PS}, is the ideal of the polynomial ring $S=\KK[x_{1},\dots, x_{n}]$ in $n$ variables over a field $\KK$ given by
\[
I(G_{\bf w})=(x^{{\bf w}(xy)}y^{{\bf w}(xy)}\mid xy\in E(G_{\bf w})).
\]
If ${\bf w}$ is the constant function defined by ${\bf w}(e)=1$ for all $e\in E(G)$, then  $I(G_{\bf w})$ is the classical  edge ideal of the underlying graph $G$ of $G_{\bf w}$, which
has been  studied extensively in the literature \cite{Ba,BHT,FM,HT,MRW,MV,M,Zh1,Zh2}.

 Recently, there has been a surge of interest in characterizing weights for which the edge ideal of a weighted graph is Cohen-Macaulay. For example,
Paulsen and Sather-Wagstaff in \cite{PS}  classified Cohen-Macaulay  weighted graphs $G_{\bf w}$ where the underlying
graph $G$ is a  cycle, a tree, or a  complete graph.  Seyed Fakhari et al. in \cite{FSTY} continued this study, classifying a
Cohen-Macaulay  weighted graph $G_{\bf w}$ when $G$ is a very well-covered graph. Recently, Diem et al. in \cite{DMV} gave a  complete  characterization of sequentially Cohen-Macaulay edge-weighted graphs. In \cite{W},  Wei classified all Cohen-Macaulay weighted chordal graphs from a purely graph-theoretic point of view. Hien in \cite{Hi}
classified Cohen-Macaulay edge-weighted graphs $G_{\bf w}$ when $G$ has girth at least $5$.

 The study of edge ideals of weighted graphs is much more recent and consequently there are  fewer results in this direction.
In this paper,  we are interested in  the regularity   of powers of the edge ideal of a weighted  integrally closed $n$-cycle $C_{\bf w}^n$. Recall that the regularity is a central invariant associated to a homogeneous ideal  $I$. It is well known that $\reg(I^t)$ is asymptotically a linear function for $t\gg  0$, i.e., there exist constants $a$,  $b$ and a positive integer $t_0$ such that for all $t\geq t_0$, $\mbox{reg}\,(I^t)=at+b$ (see \cite{CHT,K}). In this context, it has been of interest to find the exact form of this linear function and to determine  the stabilization index $t_0$ at  which $\reg(I^t)$ becomes linear (see \cite{Ba,BHT}). It turns out that even in the case of monomial ideals it is difficult to find the linear function and $t_0$ (see \cite{Con}).

Integral closure and normality of monomial ideals  play an important role in the multiplicity theory and it is also an interesting topic.
In \cite{DZCL}, Duan and the first three authors of this article gave a  complete  characterization of an integrally closed edge-weighted graph  $G_{\bf w}$ and  showed that if its underlying
graph $G$ is a  star graph,  a  path, or a cycle, then  $G_{\bf w}$ is normal. The regularity  of powers of edge ideals  has been  computed for some weighted graphs, such as star graphs, integrally closed paths, and integrally closed trees(see \cite{ZDCL,LZD}).
In this article, we will give exact formulas for the regularity of powers of edge ideals of a weighted  integrally closed cycle. Our main results are as follows.

\begin{Theorem}
  Let  $n\ge 3$ be an integer and $C_{{\bf w}}^n$  be an integrally closed  $n$-cycle, then
 \[
\reg (S/I(C_{{\bf w}}^n)^t)=2\omega t+\lfloor \frac{n}{3} \rfloor-2
\]
where $\omega=\max\{{\bf w}_1,\ldots,{\bf w}_n\}$.
\end{Theorem}

The article is organized as follows. In Section \ref{sec:prelim}, we collect  some of essential definitions and terminology needed later.
In Section \ref{sec:cycle},  by choosing different exact sequences and repeatedly using Lemma \ref{exact},  we  give some exact formulas for the regularity  of the edge ideal of a non-trivial integrally closed  $n$-cycle.
 In Section \ref{sec:power},  when $C_{\bf w}^n$ is an $n$-cycle with the weight of only one edge is non-trivial, we provide a special order on the set of minimal monomial generators of the powers of its edge ideal.
Using this order and some exact sequences, we give formulas for the regularity of the powers of these edge ideals.

\section{Preliminaries}
\label{sec:prelim}

In this section, we provide the definitions and basic facts that  will be used throughout this paper.
For detailed information, we refer to \cite{BH} and \cite{HH}.

 Let $G$ be a simple graph with the vertex set $V(G)$ and the edge set $E(G)$. For any subset $A$ of $V(G)$, $G[A]$ denotes the \emph{induced subgraph} of $G$ on the  set $A$, i.e.,
  for any $u,v \in A$, $uv \in E(G[A])$ if and only if $uv\in E(G)$.
  For each vertex $v\in V(G)$, its \emph{neighborhood} is defined as $N_G(v)\!:=\{u \in V(G)\mid uv\in E(G)\}$.
A connected graph $G$ is called a \emph{cycle} if $\deg_{G}(v)=2$ for all $v\in V(G)$. A cycle with $n$ vertices is denoted by $C_n$. A connected graph with vertex set $\{x_1,\ldots,x_n\}$, is called a \emph{path}, if $E(G)=\{x_ix_{i+1} \mid 1\leq i\leq n-1\}$. Such a path is usually denoted by $P_n$.

Given a weighted graph $G_{\bf w}$,  we denote its vertex and edge sets by $V(G_{\bf w})$ and $E(G_{\bf w})$, respectively. Any concept valid for graphs automatically applies to weighted graphs.
 For example, the \emph{neighborhood} of a vertex $v$ in a weighted graph $G_{\bf w}$ with the underlying graph  $G$ is defined as  $N_G(v)\!:=\{u \in V(G)\mid uv\in E(G)\}$. Given  a subset $W$ of $V(G_{\bf w})$, its \emph{neighborhood} is defined as  $N_G(W)\!:=\bigcup\limits_{v\in W} N_G(v)$.

The \emph{induced subgraph} on the set $A$ in $G_{\bf w}$ is the graph $G_{\bf w}[A]$ with vertex set $A$, and for any  $u,v\in A$,
 $uv\in E(G_{\bf w}[A])$ if and only if $uv\in E(G_{\bf w})$, and the weight function ${\bf w}'$ satisfies ${\bf w}'(uv)={\bf w}(uv)$.
A weighted  graph is said to be  \emph{non-trivial} if there is at least one edge with  a weight  greater than $1$, otherwise, it is said to be
{\em trivial}. An edge $e \in E(G_{\bf w})$ is said to be {\em non-trivial} if its weight ${\bf w}(e) \ge2$. Otherwise, it is said to be {\em trivial}.

\medskip
For any homogeneous ideal $I$ of the polynomial ring  $S=k[x_{1},\dots,x_{n}]$,
there exists a  finite minimal graded free resolution
$$0\rightarrow \bigoplus\limits_{j}S(-j)^{\beta_{p,j}(I)}\rightarrow \bigoplus\limits_{j}S(-j)^{\beta_{p-1,j}(I)}\rightarrow \cdots\rightarrow \bigoplus\limits_{j}S(-j)^{\beta_{0,j}(I)}\rightarrow I\rightarrow 0,$$
where $p\leq n$, and $S(-j)$ is the $S$-module obtained by shifting
the degrees of $S$ by $j$. The number
$\beta_{i,j}(I)$, the $(i,j)$-th graded Betti number of $I$, is
an invariant of $I$ that equals the number of minimal generators of degree $j$ in the
$i$th syzygy module of $I$. The regularity of $I$,  denoted by  $\mbox{reg}\,(I)$, is defined as
$$\mbox{reg}\,(I):=\mbox{max}\,\{j-i\ |\ \beta_{i,j}(I)\neq 0\},$$
which measures the size of the minimal graded
free resolution of $I$.
By looking at the minimal free resolution, it is easy to see that $\reg(I)=\reg(S/I)+1$.

The following lemmas are often used for computing  regularity  of a module.

\begin{Lemma}
	\label{product}{\em (\cite[Theorem 2.5]{CH})}
Let $R=\KK[x_{1},\dots,x_{n}]$ be a polynomial ring and  $I\subset S$ be a graded ideal with $\dim(R/I)\leq1$. Then $\reg(IM)\leq\reg(I)+\reg(M)$ for any finitely generated graded $R$-module $M$.
\end{Lemma}

\begin{Lemma}  {\em (\cite[Lemmas 2.1 and 3.1]{HT})}
	\label{exact}
Let $0\longrightarrow M\longrightarrow N\longrightarrow P\longrightarrow 0$ be a short exact
	sequence of finitely generated graded S-modules. Then we have
	\begin{itemize}
		\item[(1)] $\reg(N)\leq max\{\reg(M), \reg(P)\}$, the equality holds if $\reg(P) \neq \reg(M)-1$.
	\item[(2)] $\reg(P)\leq max\{\reg(M)-1, \reg(N)\}$, the  equality holds if $\reg(M) \neq \reg(N)$.
	\end{itemize}
\end{Lemma}

Let $\mathcal{G}(I)$ denote the unique minimal set of monomial generators of a monomial ideal $I\subset S$ and let $u\in S$
be a monomial. We set $\supp(u)=\{x_i\!: x_i |u\}$. If $G(I)=\{u_1,\ldots,u_m\}$, we set $\supp(I)=\cup_{i=1}^{m}\supp(u_i)$.

\begin{Lemma}{\em (\cite[Lemma 2.2, Lemma 3.2]{HT})}
\label{sum1}
Let $S_{1}=\KK[x_{1},\dots,x_{m}]$ and $S_{2}=\KK[x_{m+1},\dots,x_{n}]$ be two polynomial rings  over a field $\KK$. Let $S=S_1\otimes_\KK S_2$ and  $I\subset S_{1}$,
$J\subset S_{2}$ be two nonzero homogeneous  ideals.  Then
\begin{itemize}
\item[(1)] $\reg(S/(I+J))=\reg(S_1/I)+\reg(S_2/J)$,
\item[(2)]$\reg(S/JI)=\reg(S_1/I)+\reg(S_2/J)+1$.
\end{itemize}
In particular,  if  $u$ is a monomial of degree  $d$ such that $\mbox{supp}\,(u)\cap \mbox{supp}\,(I)=\emptyset$ and $J=(u)$,  then  $\mbox{reg}\,(J)=d$ and $\mbox{reg}\,(JI)=\mbox{reg}\,(I)+d$.
\end{Lemma}

\begin{Lemma}{\em (\cite[Theorem 2.5]{CH})}
\label{var}
	Let $I\subseteq S$ be a monomial ideal. Then, for any variable $x$, $\reg(I, x) \leq \reg(I)$.
\end{Lemma}

\begin{Lemma}{\em (\cite[Lemma 4.6]{ZDCL})}
\label{path}
Let $P_{\bf w}^n$ be a trivial path  with $n$ vertices, then $\reg(S/I(P^n_{\bf w}))=\lfloor \frac{n+1}{3} \rfloor$ and
    $\reg(S/I(P^n_{\bf w})^t)=2(t-1)+\reg(S/I(P^n_{\bf w}))$ for all $t\ge 1$.
\end{Lemma}

 \begin{Lemma}\label{regular}{\em (\cite[Lemma 4.4]{BHT})}
 	Let $u_1,\ldots, u_m$ be a regular sequence of homogeneous polynomials in $S$ with  $\deg\,(u_1)=\cdots=\deg\,(u_1)=d$. Let $I=(u_1,\ldots, u_m)$. Then we have $\reg\,(I^t)=dt+(d-1)(m-1)$  for all $t\ge 1$.
 \end{Lemma}

\section{The regularity of the edge ideal}
\label{sec:cycle}

Calculating or even estimating the regularity  of edge ideals of weighted cycles with arbitrary weight functions is a  difficult problem. In this section, we will give precise formulas for the  regularity of the edge ideal of  a non-trivial integrally closed  $n$-cycle.


\begin{Definition}{\em (\cite[Definition 1.1]{FHT})}\label{bettispliting}
		Let $I$ be a monomial ideal. If there exist monomial ideals $J$ and $K$ such that $\mathcal{G}(I)=\mathcal{G}(J)\cup\mathcal{G}(K)$ and $\mathcal{G}(J)\cap\mathcal{G}(K)=\emptyset$. Then $I=J+K$ is a Betti splitting if
\[
		\beta_{i, j}(I)=\beta_{i, j}(J)+\beta_{i, j}(K)+\beta_{i-1, j}(J \cap K) \text { for all } i, j \geq 0,
\]
		where $\beta_{i-1, j}(J \cap K)=0$ for $i=0$.
\end{Definition}

\begin{Lemma}{\em (\cite[Corollary 2.7]{FHT})}\label{spliting}
Suppose  $I=J+K$, where $\mathcal{G}(J)$ contains all the generators of $I$ that are divisible by some variable $x_i$, and $\mathcal{G}(K)$ is a nonempty set containing the remaining generators of $I$. If $J$ has a linear resolution, then $I=J+K$ is a Betti splitting.

Definition \ref{bettispliting} states that $\reg(I)=\max\{\reg(J), \reg(K), \reg(J \cap K)-1\}$,  as a result of the Betti splitting.
\end{Lemma}

\begin{Definition}  {\em (\cite[Definition 2.1]{F})}\label{polarization}
Let $I\subset S$ be a monomial ideal with $\mathcal{G}(I)=\{u_1,\ldots,u_m\}$, where $u_i=\prod\limits_{j=1}^n x_j^{a_{ij}}$ for $i=1,\ldots,m$.
The polarization of $I$, denoted by $I^{\calP}$, is a squarefree monomial ideal in the polynomial ring $S^{\calP}$
$$I^{\calP}=(\calP(u_1),\ldots,\calP(u_m))$$
where $\calP(u_i)=\prod\limits_{j=1}^n \prod\limits_{k=1}^{a_{ij}} x_{jk}$ is a squarefree monomial  in $S^{\calP}=\KK[x_{j1},\ldots,x_{ja_j}\mid j=1,\ldots,n]$ and $a_j=\max\{a_{ij}| i=1,\ldots,m\}$ for  $1\leq j\leq n$.
\end{Definition}

A monomial ideal  and its polarization  share many homological and
algebraic properties.  The following is a  useful property of the polarization.

\begin{Lemma}{\em (\cite[Corollary 1.6.3]{HH})}\label{polar}
 Let $I\subset S$ be a monomial ideal and $I^{\calP}\subset S^{\calP}$ be its polarization.
Then
\[
\beta_{ij}(I)=\beta_{ij}(I^{\calP})
\]
 for all $i$ and $j$. In particular, $\reg(I)=\reg(I^{\calP})$.
\end{Lemma}
		
\begin{Definition}{\em (\cite[Definition 1.4.1]{HH})}\label{integrallyclosed}
 Let  $I$  be an ideal in a ring $R$.
    An element $f \in R$ is  said to be {\em integral} over $I$ if there exists an equation
    \[
    f^k+c_1f^{k-1}+\dots+c_{k-1}f+c_k=0 \text{\ \ with\ \ }c_i \in I^i.
    \]
   The set $\overline{I}$ of elements in $R$  that are integral over $I$  is the \emph{integral closure} of $I$.
If $I=\overline{I}$, then $I$  is said to be  {\em integrally closed}.
    An edge-weighted graph $G_{\bf w}$ is said to be  {\em integrally closed} if its edge ideal $I(G_{\bf w})$ is integrally closed.
    \end{Definition}

According to \cite[Theorem 1.4.6]{HH}, every trivial graph $G_{\bf w}$ is integrally closed. The following lemma gives a complete characterization of a non-trivial graph that is integrally closed.

 \begin{Lemma}{\em (\cite[Theorem 3.6]{DZCL})}\label{integral}
 If  $G_{\bf w}$ is a  non-trivial  graph, then $I(G_{\bf w})$ is  integrally closed if and only if  $G_{\bf w}$  does not contain  any of the following three graphs as induced subgraphs.
    \begin{enumerate}
    \item  A path  $P_{\bf w}^3$ of length $2$ where  all edges have non-trivial weights.
    \item The  disjoint union $P_{\bf w}^2\sqcup P_{\bf w}^2$ of two paths $P_{\bf w}^2$  of length $1$ where all edges have non-trivial weights.
    \item A $3$-cycle $C_{\bf w}^3$ where all edges have non-trivial weights.
\end{enumerate}
\end{Lemma}

From the  above lemma, we can derive

\begin{Corollary} \label{cycleinteg}
Let $C_{\bf w}^n$ be a  non-trivial integrally closed cycle  with $n$ vertices, then if $n=6$, it can have  at most three edges with non-trivial weights; otherwise, it can have up to two edges with non-trivial weights.
\end{Corollary}

The following lemmas provide some exact formulas for the regularity of the powers of the edge ideal of a non-trivial integrally closed path.
\begin{Lemma}{\em (\cite[Theorems 4.8 and 4.9]{ZDCL}}{\em )}
\label{nontrivialpath1}
Let $P_{\bf w}^n$  be a non-trivial integrally closed  path  with  $V(P_{\bf w}^n)=\{x_1,\ldots,x_n\}$. Let $\omega_i=\max\{{\bf w}_t\mid{\bf w}_t={\bf w}(e_t)\text{\ and }  e_t=x_tx_{t+1}\text{\ for}\\
	\text{each } 1\le t\le n-1\}$, with   $\omega_i\ge 2$ and $\omega_i\ge {\bf w}_{i+2}$ if $e_{i+2}\in E({P_{\bf w}^n})$. Then,
 \begin{itemize}
\item[(1)] if $n\le 4$, then  $\reg (S/I(P_{\bf w}^n))=2\omega_i-1$;
\item[(2)]  if $n\ge 5$, then  $\reg (S/I(P_{\bf w}^n))=\max \{2\omega_i+\lfloor \frac{i-1}{3} \rfloor+\lfloor \frac{n-(i+1)}{3} \rfloor, 2{\bf w}_{i+2}+\lfloor \frac{i-2}{3} \rfloor+\lfloor \frac{n-i}{3} \rfloor\}-1$.
 \end{itemize}
\end{Lemma}

\begin{Lemma}{\em (\cite[Theorem 4.8]{ZDCL}}{\em )}
\label{nontrivialpath2}
Let $P_{\bf w}^n$  be a non-trivial integrally closed  path  with  $V(P_{\bf w}^n)=\{x_1,\ldots,x_n\}$. Let $\omega_i=\max\{{\bf w}_s\mid{\bf w}_s={\bf w}(e_s)\text{\ and }  e_s=x_sx_{s+1}\text{\ for each }
1\le s\le n-1\}$, with  $\omega_i\ge 2$ and $\omega_i\ge {\bf w}_{i+2}$ if $e_{i+2}\in E({P_{\bf w}^n})$. Then
$\reg(S/I(P_{\bf w}^n)^t)=2(t-1)\omega_i+\reg(S/I(P_{\bf w}^n))$  for all $t\ge 1$.
\end{Lemma}

The following lemma provides some exact formulas for the regularity of the powers of the edge ideal of a trivial integrally closed cycle.
\begin{Lemma}{\em (\cite[Theorem 5.2]{BHT}}{\em )}
\label{trivialcycle}
Let $C_{\bf w}^n$ be a trivial integrally closed cycle  with $n$ vertices, then
 \begin{itemize}
   \item[(1)]   $\reg(S/I(C_{\bf w}^n))=\lfloor \frac{n+1}{3} \rfloor$.
 \item[(2)]   $\reg(S/I(C^n_{\bf w})^t)=2(t-1)+\reg(S/I(C^n_{\bf w}))$ for all $t\ge 1$.
 \end{itemize}
\end{Lemma}

Thus, we will now consider a non-trivial integrally closed cycle that satisfies the following conditions.
\begin{Remark}
    \label{n-cycle}
    Let  $n\ge 3$ be an integer and $C_{\bf w}^n$  be a non-trivial integrally closed  $n$-cycle  with  $V(C_{\bf w}^n)=\{x_1,\ldots,x_n\}$ and  $E(C_{\bf w}^n)=\{e_1,\ldots,e_{n}\}$, where $e_i=x_ix_{i+1}$ with $x_{n+1}=x_1$ and ${\bf w}_i={\bf w}(e_i)$ for all $i \in [n]$.
\end{Remark}

We first compute the regularity  of  the edge ideal of a cycle $C_{\bf w}^n$  for $3\leq n\leq 6$.

\begin{Theorem}
\label{smalln} Let $C_{\bf w}^n$ be an $n$-cycle as in Remark \ref{n-cycle}, where $3\leq n\leq 6$. Let $\omega=\max\{{\bf w}_1,\ldots,{\bf w}_{n}\}$. Then
\[
\reg (S/I(C_{\bf w}^n))=2\omega+\lfloor \frac{n}{3} \rfloor-2.
\]
\end{Theorem}
\begin{proof}  Let $I=I(C_{\bf w}^n)$ and $\omega={\bf w}_1$  by symmetry. If $n=3$, then ${\bf w}_2=1$ or ${\bf w}_3=1$ by Lemma \ref{integral}. We can assume $\bf w_3=1$ due to symmetry. So $I:x_3=(x_2^{{\bf w}_2}x_3^{{\bf w}_2-1},x_1)$ and $(I,x_3)=(x_1^{{\bf w}_1}x_2^{{\bf w}_1},x_3)$. It follows from Lemma \ref{sum1} that
		$\reg(S/(I:x_3))=2{\bf w}_2-2$ and  $\reg(S/(I,x_3))=2{\bf w}_1-1$. Applying  Lemma \ref{exact} to the following short exact sequence with $i=3$
\begin{gather}\label{eqn: equality1}
\begin{matrix}
 0 & \rightarrow & \frac{S}{I : x_{i}}(-1)  & \stackrel{ \cdot x_i} \longrightarrow  & \frac{S}{I} & \rightarrow & \frac{S}{(I,x_{i})} & \rightarrow & 0,
 \end{matrix}
\end{gather}
we can observe that $\reg (S/I)=2{\bf w}_1-1$.

If $n=4$, then ${\bf w}_2={\bf w}_4=1$  by Lemma \ref{integral}. Thus $(I : x_4)=(x_2x_3,x_3^{{\bf w}_3}x_4^{{\bf w}_3-1},x_1)$ and  $(I,x_4)=(x_1^{{\bf w}_1}x_2^{{\bf w}_1},x_2x_3,x_4)$. Let
$(I : x_4)^{\calP}$ and $(I,x_4)^{\calP}$ be  the polarizations of ideals $I : x_4$ and $(I,x_4)$, respectively. By Lemmas \ref{sum1}, \ref{spliting} and \ref{polar},
$\reg(S/(I : x_4))=\reg(S^{\calP}/(I : x_4)^{\calP})=2{\bf w}_3-2$ and $\reg(S/(I,x_4))=\reg(S^{\calP}/(I,x_4)^{\calP})=2{\bf w}_1-1$.
 Using  Lemma \ref{exact} and  the  short exact sequence (\ref{eqn: equality1}) with $i=4$,
we can see that $\reg (S/I)=2{\bf w}_1-1$.

If $n=5$, then ${\bf w}_2={\bf w}_4={\bf w}_5=1$  by symmetry and Lemma \ref{integral}. Thus $(I : x_5)=(x_2x_3,x_1,x_4)$ and  $(I,x_5)=(x_1^{{\bf w}_1}x_2^{{\bf w}_1},x_2x_3,x_3^{{\bf w}_3}x_4^{{\bf w}_3},x_5)$. It follows from Lemmas \ref{sum1} and  \ref{nontrivialpath1} that $\reg(S/(I : x_5))=1$ and $\reg(S/(I,x_5))=2{\bf w}_1-1$. This implies that   $\reg (S/I)=2{\bf w}_1-1$ due to
 Lemma \ref{exact} and  the  short exact sequence (\ref{eqn: equality1}) with $i=5$.

If $n=6$,  then ${\bf w}_1\ge {\bf w}_3\ge {\bf w}_5$ and ${\bf w}_2={\bf w}_4={\bf w}_6=1$ by symmetry and Lemma \ref{integral}. Thus $(I,x_6)=(x_1^{{\bf w}_1}x_2^{{\bf w}_1},x_2x_3,x_3^{{\bf w}_3}x_4^{{\bf w}_3},x_4x_5,x_6)$ and  $(I : x_6)=(x_2x_3,x_3^{{\bf w}_3}x_4^{{\bf w}_3},$ $x_4x_5,x_5^{{\bf w}_5}x_6^{{\bf w}_5-1},x_1)$. By Lemma \ref{nontrivialpath1}, $\reg(S/(I,x_6))=2{\bf w}_1$. To compute $\reg(S/(I : x_6))$, we distinguish
between the following two cases:

(i) If ${\bf w}_5=1$, then $(I:x_6)=(x_2x_3,x_3^{{\bf w}_3}x_4^{{\bf w}_3},x_5,x_1)$. It follows from Lemmas \ref{sum1} and \ref{nontrivialpath1} that $\reg(S/(I : x_6))=2{\bf w}_3-1$.

(ii) If ${\bf w}_5\ge 2$, then $((I:x_6):x_5)=(x_2x_3,x_5^{{\bf w}_5-1}x_6^{{\bf w}_5-1},x_4,x_1)$ and $((I:x_6),x_5)=(x_2x_3,x_3^{{\bf w}_3}x_4^{{\bf w}_3},x_5,x_1)$. It follows from Lemmas \ref{sum1} and \ref{nontrivialpath1} that $\reg(S/((I:x_6):x_5))=2{\bf w}_5-2$ and $\reg(S/((I:x_6),x_5))=2{\bf w}_3-1$. Thus, by  Lemma \ref{exact} and  the following short exact sequence
\begin{gather*}
\begin{matrix}
 0 & \rightarrow & \frac{S}{(I:x_{6}) : x_{5}}(-1)  & \stackrel{ \cdot x_5} \longrightarrow  & \frac{S}{I:x_{6}} & \rightarrow & \frac{S}{((I:x_{6}),x_{5})} & \rightarrow & 0,
 \end{matrix}
\end{gather*}
we get  $\reg (S/(I : x_6))=2{\bf w}_3-1$.

In any case, again using Lemma \ref{exact} and  the exact sequence (\ref{eqn: equality1}) with $i=6$,
we can determine that $\reg (S/I)=2{\bf w}_1$.
\end{proof}

\begin{Theorem}
\label{bign} Let $n\geq 7$ be an integer and let $C_{\bf w}^n$ be an $n$-cycle as in Remark \ref{n-cycle}.  Then
\[
	\reg (S/I(C_{\bf w}^n))=2{\omega}+\lfloor \frac{n}{3} \rfloor-2
	\]
where  $\omega=\max\{{\bf w}_1,\ldots,{\bf w}_{n}\}$.
\end{Theorem}
\begin{proof}  Let $I=I(C_{\bf w}^n)$ and  $\omega={\bf w}_1$,  then we can assume  ${\bf w}_1\ge {\bf w}_3$ and other ${\bf w}_i=1$ by symmetry and Lemma \ref{integral}.
		Thus $(I : x_n)=I(P_{\bf w}^{n-3})+(x_1,x_{n-1})$ and  $(I,x_n)=I(P_{\bf w}^{n-1})+(x_n)$, where $P_{\bf w}^{n-3}$ and $P_{\bf w}^{n-1}$ are  induced subgraphs of $C_{\bf w}^n$ on the sets $\{x_2,\ldots,x_{n-2}\}$
and $\{x_1,\ldots,x_{n-1}\}$, respectively.
It follows from Lemmas  \ref{sum1} and \ref{nontrivialpath1} that
\begin{eqnarray*}
\reg(S/(I,x_n))&=&\max \{2{\bf w}_1+\lfloor \frac{1-1}{3} \rfloor+\lfloor \frac{n-1-(1+1)}{3} \rfloor, 2{\bf w}_{1+2}+\lfloor \frac{1-2}{3} \rfloor\\
&+&\lfloor \frac{n-1-1}{3} \rfloor\}-1\\
&=&\max \{2{\bf w}_1+\lfloor \frac{n}{3} \rfloor-1, 2{\bf w}_{3}+\lfloor \frac{n-2}{3} \rfloor-1\}-1\\
&=&2{\bf w}_1+\lfloor \frac{n}{3} \rfloor-2.
\end{eqnarray*}
On the other hand, if ${\bf w}_3=1$, then $\reg(S/(I:x_n))=\lfloor \frac{n-2}{3} \rfloor$ by Lemma \ref{path}. Otherwise, by Lemmas  \ref{sum1} and \ref{nontrivialpath1}, we get that
\begin{eqnarray*}
\reg(S/(I:x_n))&=&\max \{2{\bf w}_3+\lfloor \frac{2-1}{3} \rfloor+\lfloor \frac{n-3-(2+1)}{3} \rfloor, 2+\lfloor \frac{2-2}{3} \rfloor\\
&+&\lfloor \frac{n-3-2}{3} \rfloor\}-1\\
&=&\max \{2{\bf w}_3+\lfloor \frac{n}{3} \rfloor-2, \lfloor \frac{n-2}{3} \rfloor+1\}-1\\
&=&2{\bf w}_3+\lfloor \frac{n}{3} \rfloor-3.
\end{eqnarray*}
Applying Lemma \ref{exact}  to the  short exact sequence (\ref{eqn: equality1}) with $i=n$, we can determine that $\reg (S/I)=2{\bf w}_1+\lfloor \frac{n}{3} \rfloor-2$.
\end{proof}

\section{The regularity of powers of the edge ideal}
\label{sec:power}
In this section, we will give precise formulas for the  regularity of the powers of the edge ideal of  a non-trivial integrally closed  $n$-cycle.

We will now  analyze the regularity  of  $I(C_{\bf w}^n)^t$ for all $t\ge 2$.
\begin{Lemma}\label{cyclecolon}
Let $C_{\bf w}^n$ be an $n$-cycle as in Remark \ref{n-cycle}. If ${\bf w}_{i}=1$, then
\[
((I(C^n_{\bf w})^t: x_i),x_{i+1})=((I(C^n_{\bf w} \setminus {x_{i+1}})^t: x_i),x_{i+1})
\]
for all $t\ge 2$.
\end{Lemma}
\begin{proof} It is clear that $((I(C^n_{\bf w} \setminus {x_{i+1}})^t: x_i),x_{i+1})\subseteq ((I(C^n_{\bf w})^t: x_i),x_{i+1})$.  For any monomial $u \in \mathcal{G}(((I(C^n_{\bf w})^t: x_i),x_{i+1}))$, we have $u \in ((I(C^n_{\bf w} \setminus {x_{i+1}})^t: x_i),x_{i+1})$  if
 $x_{i+1}|u$. Otherwise,  $ux_i \in I(C_{\bf w}^n)^t$, so we  can write $ux_i$ as
\[
ux_i=u_{\ell 1} \cdots u_{\ell t}v
\]
  where each $u_{\ell j} \in \mathcal{G}(I(C^n_{\bf w}\setminus {x_{i+1}}))$ and $v$ is a monomial, which means that $u \in \mathcal{G}((I(C^n_{\bf w} \setminus {x_{i+1}})^t: x_i))$.
\end{proof}

\begin{Theorem}
\label{3-cycle} Let $C_{\bf w}^3$ be a $3$-cycle as in Remark \ref{n-cycle}.  Then, for all $t\ge 1$,
\[
\reg (S/I(C_{\bf w}^3)^t)=2\omega t-1
\]
where $\omega=\max\{{\bf w}_1,{\bf w}_2,{\bf w}_3\}$.
\end{Theorem}
\begin{proof} Let $I=I(C_{\bf w}^3)$ and  $\omega={\bf w}_1$,  then ${\bf w}_2=1$ or ${\bf w}_3=1$ by Lemma \ref{integral}. We can assume ${\bf w}_3=1$ due to symmetry.
 We distinguish between the following two cases:

(i) If ${\bf w}_2=1$,  then $I^t=((x_1x_2)^{{\bf w}_1}+Jx_3)^{t}=((x_1x_2)^{t{\bf w}_1})+\sum\limits_{i=1}^t(x_1x_2)^{{\bf w}_1(t-i)}(Jx_3)^i$, where $J=(x_1,x_{2})$.
	In this case, we set  $P_0=I^t$, $P_j=I^t:x_3^j$ and $Q_{j-1}=P_{j-1}+(x_3)$ for all $j\in [t]$, then $P_t=J^t$, $P_j=(x_1x_2)^{{\bf w}_1t}+\sum\limits_{i=1}^{j-1}(x_1x_2)^{{\bf w}_1(t-i)}J^i+(x_1x_2)^{{\bf w}_1 (t-j)}J^j+\sum\limits_{i=j+1}^t(x_1x_2)^{{\bf w}_1(t-i)}J^ix_3^{i-j}=(x_1x_2)^{{\bf w}_1 (t-j)}J^j+\sum\limits_{i=j+1}^t(x_1x_2)^{{\bf w}_1(t-i)}J^ix_3^{i-j}$
		for all $j\in [t-1]$, and $Q_{j-1}=(x_1x_2)^{{\bf w}_1 (t-j+1)}J^{j-1}+(x_3)$ for all $j\in [t]$. Thus $\reg(S/Q_0)=\reg(Q_0)-1=2t{\bf w}_1-1$   and $\reg(S/P_t)=\reg(S/J^t)=t-1$ by Lemma \ref{regular}.
Note that $\dim(S/J^j)=1$	for all $j\in[t-1]$. Then, by Lemmas \ref{product}, and  \ref{sum1}(1), we have
	\begin{align*}
		\reg(S/Q_j) &=\reg((x_1x_2)^{{\bf w}_1 (t-j)}J^{j})-1\\
		&\leq\reg((x_1x_2)^{{\bf w}_1 (t-j)})+\reg(J^j)-1\\
		&=2{\bf w}_1(t-j)+j-1\\
		&\leq 2{\bf w}_1t-j-1.
	\end{align*}
	 Using  Lemma \ref{exact} to the following exact sequences
	\begin{gather*}
		\hspace{1.0cm}\begin{matrix}
			0 & \longrightarrow & \frac{S}{P_1}(-1)  & \stackrel{ \cdot x_3} \longrightarrow  &  \frac{S}{I^t} & \longrightarrow &\frac{S}{Q_0} & \longrightarrow & 0,\\
			0 & \longrightarrow & \frac{S}{P_2}(-1) & \stackrel{ \cdot x_3} \longrightarrow & \frac{S}{P_1} &\longrightarrow & \frac{S}{Q_1} & \longrightarrow & 0,  \\
			&  &\vdots&  &\vdots&  &\vdots&  &\\
			0 & \longrightarrow & \frac{S}{P_{t-1}}(-1) & \stackrel{ \cdot x_3} \longrightarrow & \frac{S}{P_{t-2}} &\longrightarrow & \frac{S}{Q_{t-2}} & \longrightarrow & 0,  \\
			0 & \longrightarrow & \frac{S}{P_t}(-1) & \stackrel{ \cdot x_3} \longrightarrow & \frac{S}{P_{t-1}} &\longrightarrow & \frac{S}{Q_{t-1}} & \longrightarrow & 0,
		\end{matrix}
	\end{gather*}
	we get the desired results.

(ii) If ${\bf w}_2\geq 2$, then $(I^t: x_1x_3)=I^{t-1}$. In fact,  for any monomial $u \in \mathcal{G}(I^t: x_1x_3)$, we have $ux_1x_3 \in I^t$. We can write $ux_1x_3$ as $ux_1x_3=(x_1^{{\bf w}_1}x_2^{{\bf w}_1})^i(x_2^{{\bf w}_2}x_3^{{\bf w}_2})^j(x_1x_3)^{k}h$,
where $i+j+k=t$. If $k\geq 1$, then $u=(x_1^{{\bf w}_1}x_2^{{\bf w}_1})^i(x_2^{{\bf w}_2}x_3^{{\bf w}_2})^j(x_1x_3)^{k-1}h\in I^{t-1}$.  Otherwise, $i+j=t$. It is obvious that  $u\in I^{t-1}$ if $i=0$ or $j=0$. We can suppose
 $i,j\geq 1$. In this case, $u=(x_1^{{\bf w}_1}x_2^{{\bf w}_1})^{i-1}(x_2^{{\bf w}_2}x_3^{{\bf w}_2})^{j-1}(x_1x_3)(x_1^{{\bf w}_1-2}x_2^{{\bf w}_1})(x_2^{{\bf w}_2}x_3^{{\bf w}_2-2})h\in I^{t-1}$ since ${\bf w}_1\geq 2$ and ${\bf w}_2\geq 2$.

Now we prove  $\reg(S/I^t)=2{\bf w}_1t-1$ by induction on $t$. The case where $t=1$ is verified by  Theorem \ref{smalln}. In the following, we assume that $t\geq 2$. By Lemma \ref{cyclecolon}, we have  $((I^t : x_3),x_1)=((x_2x_3)^{{\bf w}_2}:x_3)+(x_1)=(x_2^{{\bf w}_2}x_3^{{\bf w}_2-1})+(x_1)$ and $(I^t,x_3)=((x_1x_2)^{{\bf w}_1t},x_3)$, thus $\reg(S/((I^t : x_3),x_1))=2{\bf w}_2t-2$, $\reg(S/(I^t ,x_3))=2{\bf w}_1t-1$ and $\reg(S/(I^t : x_1x_3))=2{\bf w}_1(t-1)-1$ by the inductive hypothesis.
	Using
 Lemma \ref{exact} to the following short exact sequences
\begin{gather*}
\hspace{1cm}\begin{matrix}
 0 & \longrightarrow & \frac{S}{I^t : x_3}(-1)  & \stackrel{ \cdot x_3} \longrightarrow  & \frac{S}{I^t} & \longrightarrow & \frac{S}{(I^t,x_3)} & \longrightarrow & 0,\\
 0 & \longrightarrow & \frac{S}{I^t : x_1x_3}(-1)  & \stackrel{ \cdot x_1} \longrightarrow  & \frac{S}{I^t : x_3} & \longrightarrow & \frac{S}{((I^t : x_3),x_1)} & \longrightarrow & 0,
 \end{matrix}
\end{gather*}
we can determine that $\reg (S/I^t)=2{\bf w}_1t-1$.
\end{proof}

Next, we consider the case where $n\geq 4$.

\begin{Lemma}\label{cyclecolon2}
Let $n\geq 4$ be an integer and  $C_{\bf w}^n$  be an $n$-cycle as in Remark \ref{n-cycle}. If ${\bf w}_{i}, {\bf w}_{i+2}\geq 2$ and ${\bf w}_{i+1}=1$ for some $i$, then
\[
I(C^n_{\bf w})^t: x_{i+1}x_{i+2}=I(C^n_{\bf w})^{t-1}
\]
 for all $t\ge 2$.
  \end{Lemma}
\begin{proof}It is clear that $I(C^n_{\bf w})^{t-1}\subseteq I(C^n_{\bf w})^{t}:x_{i+1}x_{i+2}$.  For any monomial $u \in \mathcal{G}(I(C^n_{\bf w})^t: x_{i+1}x_{i+2})$, we have  $ux_{i+1}x_{i+2} \in I(C^n_{\bf w})^t$. Let
\begin{gather}\label{eqn: equality2}
ux_{i+1}x_{i+2}=u_{\ell 1} \cdots u_{\ell t}h
\end{gather}
  where each $u_{\ell j} \in \mathcal{G}(I(C^n_{\bf w}))$  and $h$ is a  monomial. If $x_{i+1}| h$ or $x_{i+2}| h$, then $u\in I(C^n_{\bf w})^{t-1}$ by Eq.(\ref{eqn: equality2}). Otherwise, $x_{i+1}| u_{\ell j_1}$ and $x_{i+2}| u_{\ell j_2}$ for some  $j_1,j_2\in[t]$. If $j_1=j_2$, then $u\in I(C^n_{\bf w})^{t-1}$ by Eq.(\ref{eqn: equality2}).  Otherwise, we have $u_{\ell j_1}=(x_{i}x_{i+1})^{{\bf w}_{i}}$ and $u_{\ell j_2}=(x_{i+2}x_{i+3})^{{\bf w}_{i+2}}$ since $N_G(x_{i+1})=\{x_{i},x_{i+2}\}$ and $N_G(x_{i+2})=\{x_{i+1},x_{i+3}\}$. It follows that $(u_{\ell j_1}u_{\ell j_2})/(x_{i+1}x_{i+2})=(x_{i+1}x_{i+2})(x_{i}^{{\bf w}_{i}}x_{i+1}^{{\bf w}_{i}-2})(x_{i+2}^{{\bf w}_{i+2}-2}x_{i+3}^{{\bf w}_{i+2}})\in I(C^n_{\bf w})$ since ${\bf w}_{i},{\bf w}_{i+2}\geq 2$. This implies that
$u\in I(C^n_{\bf w})^{t-1}$ by Eq.(\ref{eqn: equality2}).
\end{proof}

\begin{Theorem}
\label{n-cycle2} Let $n\geq 4$ be an integer and  $C_{\bf w}^n$  be an $n$-cycle as in Remark \ref{n-cycle}.
If there exist at least two edges with non-trivial weights, then, for all $t\ge 1$,
\[
\reg (S/I(C_{\bf w}^n)^t)=2\omega t+\lfloor \frac{n}{3} \rfloor-2
\]
where $\omega=\max\{{\bf w}_1,\ldots,{\bf w}_n\}$.
\end{Theorem}
 \begin{proof} Let $I=I(C_{\bf w}^n)$ and  $\omega={\bf w}_1$, then we can suppose ${\bf w}_1\ge {\bf w}_3\ge 2$ and  ${\bf w}_2=1$. We prove the statement by induction on $t$.  The case where $t=1$ is verified by Theorems \ref{smalln} and  \ref{bign}. Now, we assume that $t\geq 2$. Since $(I^t,x_3)=(I(C_{\bf w}^n\setminus x_3)^t,x_3)$ and $(I^t : x_2x_3)=I^{t-1}$ by Lemma \ref{cyclecolon2},
 it follows from Lemmas \ref{nontrivialpath1}, \ref{nontrivialpath2} and the inductive hypothesis that
  \begin{align*}
 \reg(S/(I^t,x_3))&=2{\bf w}_1t+\lfloor \frac{n}{3} \rfloor-2,\\
 \reg(S/(I^t : x_2x_3))&=2{\bf w}_1(t-1)+\lfloor \frac{n}{3} \rfloor-1.
 \end{align*}
On the other hand,  by Lemma \ref{cyclecolon}, we have  $((I^t : x_3),x_2)=(I(P_{\bf w}^{n-1})^t:x_3)+(x_2)$ and $(I(P_{\bf w}^{n-1})^t,x_3)=I(P_{\bf w}^{n-2})^t+(x_3)$, where $P_{\bf w}^{n-1}$ and $P_{\bf w}^{n-2}$
 are  induced graphs of $C_{\bf w}^n$ on the sets  $\{x_1,\hat{x}_2,x_3,\ldots,x_n\}$ and  $\{x_1,\hat{x}_2,\hat{x}_3,x_4,\ldots,x_n\}$, respectively, where $\hat{x}_2$ denotes the element $x_2$ which is omitted from the set $\{x_1,x_2,\ldots,x_n\}$. Thus $\reg(S/(I(P_{\bf w}^{n-1})^t,x_3))=2(t-1)+\lfloor \frac{n-1}{3} \rfloor$  and $\reg(S/I(P_{\bf w}^{n-1})^t)=2{\bf w}_3t+\lfloor \frac{n}{3} \rfloor-2$ by Lemmas \ref{path} and  \ref{nontrivialpath2}.  By Lemma \ref{exact} and the short exact sequence
 \begin{gather*}
\hspace{1cm}\begin{matrix}
 0 & \longrightarrow & \frac{S}{I(P_{\bf w}^{n-1})^t : x_3}(-1)  & \stackrel{ \cdot x_3} \longrightarrow  & \frac{S}{I(P_{\bf w}^{n-1})^t} & \longrightarrow & \frac{S}{(I(P_{\bf w}^{n-1})^t,x_3)} & \longrightarrow & 0,
 \end{matrix}
\end{gather*}
 we obtain that
 \begin{align*}
\reg(S/((I^t : x_3),x_2))&=\reg(S/(I(P_{\bf w}^{n-1})^t:x_3)+(x_2)))=\reg(S/(I(P_{\bf w}^{n-1})^t:x_3))\\
&\leq \max \{2{\bf w}_3t+\lfloor \frac{n}{3} \rfloor-2, 2(t-1)+\lfloor \frac{n-1}{3} \rfloor+1\}-1\\
&=2{\bf w}_3t+\lfloor \frac{n}{3} \rfloor-3.
\end{align*}
Again, apply  Lemma \ref{exact} to the following short exact sequences
\begin{gather*}
\hspace{1cm}\begin{matrix}
 0 & \longrightarrow & \frac{S}{I^t : x_3}(-1)  & \stackrel{ \cdot x_3} \longrightarrow  & \frac{S}{I^t} & \longrightarrow & \frac{S}{(I^t,x_3)} & \longrightarrow & 0,\\
 0 & \longrightarrow & \frac{S}{I^t : x_3x_2}(-1)  & \stackrel{ \cdot x_2} \longrightarrow  & \frac{S}{I^t : x_3} & \longrightarrow & \frac{S}{((I^t : x_3),x_2)} & \longrightarrow & 0,
 \end{matrix}
\end{gather*}
we can determine that $\reg (S/I^t)=2{\bf w}_1 t+\lfloor \frac{n}{3} \rfloor-2$.
\end{proof}

\medskip
 We now consider the case where the weight of only one edge is non-trivial. We will follow the technique of \cite{Ba,WZX}. In order to prove Theorem \ref{n-cycle3}, we will define a specific order with a nice property in the set $\mathcal{G}(I(C^n_{\bf w})^t$ of minimal monomial generators of powers of the edge ideal of $C_{{\bf w}}^n$.
 Using this order, we obtain two main results for computing the regularity of powers of the edge ideal of $C_{{\bf w}}^n$.
\begin{Setting}
    \label{1-edge-weighted}
    Let $n\geq 4$ be an integer and  $C_{\bf w}^n$ be  an $n$-cycle as in Remark \ref{n-cycle}. We can write $L_i=(x_{i}x_{i+1})^{{\bf w}_i}$ for $i=1,2,\ldots,n$, where  $x_j=x_i$ if $j\equiv i\  (\text{mod\ }n)$. We always assume by symmetry that ${\bf w}_1 \geq 2$ and ${\bf w}_i=1$ for each $i=2,3,\ldots,n$.  We stipulate an order $L_1>\cdots>L_n$ on the set $\mathcal{G}(I(C_{\bf w}^n))$. For each integer $t \geq 1$, we define an order on the set $\mathcal{G}(I(C_{\bf w}^n)^t)$ as follows:  Let $M, N \in \mathcal{G}(I(C_{\bf w}^n)^t)$ and we write $M, N$ as $M=L_1^{a_1}L_2^{a_2}\cdots L_n^{a_n}$, $N=L_1^{b_1}L_2^{b_2}\cdots L_n^{b_n}$ with $\sum\limits_{i=1}^{n}a_{i}=\sum\limits_{i=1}^{n}b_{i}=t$. We say $M>N$ if and only if  $({a_1},a_2,\ldots,{a_n})>_{lex}({b_1},b_2,\ldots,{b_n})$.  We denote by $L^{(t)}$ the totally ordered set of  $\mathcal{G}(I(C_{\bf w}^n)^t)$, ordered in the way discussed above,  and by $L_k^{(t)}$ the $k$-th element of the set $L^{(t)}$. Sometimes we denote by $L_k$ for $L_k^{(1)}$. It follows from Theorem \ref{unique} that this order is well-defined.
\end{Setting}

\begin{Theorem}\label{unique}
Under the assumptions of Setting \ref{1-edge-weighted}, for each   positive integer $t$ and $M \in \mathcal{G}(I(C_{\bf w}^n)^t)$, then  there exist unique ${a_1},a_2,\ldots,{a_n}$ such that  $M=L_{1}^{a_{1}}L_{2}^{a_{2}}\cdots L_{n}^{a_n}$ with $\sum\limits_{i=1}^{n}a_{i}=t$, where $a_i\geq 0$ for all $1\leq i\leq n$.
\end{Theorem}
\begin{proof}
Let $M=L_{1}^{a_{1}}L_{2}^{a_{2}}\cdots L_{n}^{a_n}=L_1^{b_1}L_2^{b_2}\cdots L_n^{b_n}$ be two expressions of $M$ with  $\sum\limits_{i=1}^{n}a_{i}=\sum\limits_{i=1}^{n}b_{i}=t$ and $a_i,b_i\geq 0$ for all $1\leq i\leq n$.
By comparing the degree of each variable in the expressions of $M$, we can conclude that
\[
\left\{
\begin{array}{ccccc}
{\bf w}_1(a_1-b_1)&+&(a_n-b_n)&=&0\\
{\bf w}_1(a_1-b_1)&+& (a_2-b_2)&=&0\\
 (a_2-b_2)&+& (a_3-b_3)&=&0\\
  &\vdots& &&\vdots\\
(a_{n-1}-b_{n-1})&+&(a_n-b_n)&=&0
\end{array}
	\right.
\]
Let \[
	\mathbf{A}=
	\begin{bmatrix}
		{\bf w}_1& 0 &0&\ldots&0&0&1\\
		{\bf w}_1&1&0&\ldots&0&0&0\\
		0&1&1&\ldots &0&0&0\\
		\vdots & \vdots & \vdots& \ddots& \vdots& \vdots& \vdots\\
		0&0&0&\ldots &1&1&0\\
		0&0&0&\ldots &0&1&1\\
	\end{bmatrix},
	\]
then it is easy to see that its determinant $|A|\neq 0$ if $n$ is odd. If $n$ is even, then the solution of the above system of equations with $A$ as the coefficient matrix is
 $a_1-b_1=\frac{-1}{{\bf w}_1}(a_n-b_n)$ and $a_i-b_i=(-1)^{i}(a_n-b_n)$ for each $2\le i\le n$. Thus $\sum\limits_{i=1}^{n}(a_i-b_i)=\frac{-1}{{\bf w}_1}(a_n-b_n)+\sum\limits_{i=2}^{n}(-1)^{i}(a_n-b_n)=(1-\frac{1}{{\bf w}_1})(a_n-b_n)$.
 Since $\sum\limits_{i=1}^{n}a_{i}=\sum\limits_{i=1}^{n}b_{i}$, we have $a_n=b_n$, it follows that $a_i=b_i$ for all $i\in [n]$, since $a_1-b_1=\frac{-1}{{\bf w}_1}(a_n-b_n)$ and $a_i-b_i=(-1)^{i}(a_n-b_n)$ for all $2\le i\le n$.
\end{proof}

\begin{Definition}\label{edgedivide}
	Let $M_1 \in \mathcal{G}(I(C_{{\bf w}}^n)^k)$, $M_2 \in \mathcal{G}(I(C_{{\bf w}}^n)^t)$ with $1 \leq k \leq t$. We say that $M_1$ divides $M_2$ as an edge, denoted by $M_1|^{edge}M_2$, if  there exists an element $M_3 \in \mathcal{G}(I(C_{{\bf w}}^n)^{t-k})$ such that $M_2=M_1M_3$. Otherwise, we  denote by $M_1\nmid\!^{edge}M_2$.
\end{Definition}

For a monomial $u=x_1^{a_1}\cdots x_n^{a_n}$ with $a_i\ge 0$ for all $i\in [n]$, we set $\deg_{x_i}(u)=a_i$ and $\supp(u)=\{x_i: a_i>0\ \text{for\ }i\in [n]\}$. The following two theorems are the most important technical results in the following section.

\begin{Lemma}\label{t=2}
Under the assumptions of Setting \ref{1-edge-weighted},	 we consider the case where $t=2$. If $L_{i}^{(2)}, L_{j}^{(2)}\in L^{(2)}$ with $L_{i}^{(2)}>L_{j}^{(2)}$ and $L_1|^{edge}L_{i}^{(2)}$, then there exists some $L_{k}^{(2)}\in L^{(2)}$ such that  $L_{i}^{(2)}>L_{k}^{(2)}$ with $(L_{j}^{(2)}:L_{i}^{(2)})\subseteq (L_{k}^{(2)}:L_{i}^{(2)})$ and $(L_{k}^{(2)}:L_{i}^{(2)})$ has one of the following two forms:
	\begin{itemize}
		\item[(1)] $(L_{k}^{(2)}:L_{i}^{(2)})=(L_{\ell_2}:L_{\ell_1})$ with $L_{\ell_1}>L_{\ell_2}$,  $L_{\ell_2}|^{edge}L_{k}^{(2)}$ and $L_{\ell_1}|^{edge}L_{i}^{(2)}$;
		\item[(2)] $(L_{k}^{(2)}:L_{i}^{(2)})=(x_{n-2})$ with $L_{n-2}|^{edge}L_{k}^{(2)}$, $L_n|^{edge}L_{k}^{(2)}$, $L_1|^{edge}L_{i}^{(2)}$ and $L_{n-1}|^{edge}L_{i}^{(2)}$.
	\end{itemize}
\end{Lemma}
\begin{proof}
Let $L_{j}^{(2)}=L_{j_1}L_{j_2}$ and
$L_{i}^{(2)}= L_{i_1}L_{i_2}$ with $i_1=1$,  where $1 \leq  j_1\leq j_{2} \leq n$ and $1 \leq  i_2\leq n$. If  $L_{i_a}|^{edge}L_j^{(t)}$ for some $a\in [2]$,  then we choose $L_k^{(2)}=L_{j}^{(2)}$.
Thus $(L_j^{(2)}:L_i^{(2)})=(L_{k}^{(2)}:L_{i}^{(2)})=(\frac{L_{j}^{(2)}}{L_{i_a}}:\frac{L_{i}^{(2)}}{L_{i_a}})$ with $\frac{L_{i}^{(2)}}{L_{i_a}}>\frac{L_{j}^{(2)}}{L_{i_a}}$.
 If $L_{i_a}\!\nmid^{edge}\! L_j^{(2)}$ for any $a\in[2]$, then $j_1,j_2\ge 2$. We consider the following two cases:

 (I) If $\supp(L_{j_\ell})\cap \supp(L_{i}^{(2)})=\emptyset$ for some $\ell\in [2]$, then we choose $L_k^{(2)}=\frac{L_{i}^{(2)}}{L_{i_1}} L_{j_\ell}$.  Thus $(L_j^{(2)}:L_i^{(2)})\subseteq L_{j_\ell}$ and $(L_{k}^{(2)}:L_{i}^{(2)})=(L_{j_\ell}:L_{i_1})=L_{j_\ell}$ with $L_{i_1}>L_{j_\ell}$.

 (II) If $\supp(L_{j_\ell})\cap \supp(L_{i}^{(2)})\ne \emptyset$ for all $\ell\in [2]$. The case where  $L_{j_\ell}|^{edge}L_{i}^{(2)}$ for some $\ell\in [2]$ follows from the condition that $L_{i_a}|^{edge}L_j^{(t)}$ for some $a\in [2]$. Now we assume that $L_{j_\ell}\nmid^{edge}L_{i}^{(2)}$ for all $\ell\in [2]$.
We have two subcases:

(i) If $\deg_{x_{j_{\ell}+1}}(L_j^{(2)})>\deg_{x_{j_{\ell}+1}}(L_i^{(2)})$ for  some $\ell\in[2]$, then  $L_{j_{\ell}'}|^{edge}L_{i}^{(2)}$ for some $j_{\ell}'\in\{j_{\ell}-1,j_{\ell}+1\}$ since $\supp(L_{j_{\ell}})\cap \supp(L_{i}^{(2)})\neq\emptyset$ for all ${\ell}\in[2]$. We also have two subcases:

(a) If $L_{j_{\ell}-1}|^{edge}L_{i}^{(2)}$, then we choose $L_k^{(2)}=\frac{L_{i}^{(2)}}{L_{j_{\ell}-1}}L_{j_{\ell}}$. Thus  $(L_j^{(2)}:L_i^{(2)})\subseteq (x_{j_{\ell}+1})$ and $(L_{k}^{(2)}:L_{i}^{(2)})=(L_{j_{\ell}}:L_{j_{\ell}-1})=(x_{j_{\ell}+1})$.

(b) If $L_{j_{\ell}-1}\!\nmid^{edge}\!L_{i}^{(2)}$ and  $L_{j_{\ell}+1}|^{edge}L_{i}^{(t)}$, then we choose $L_k^{(2)}=L_{j_{\ell}}L_{j_{\ell}+1}$. Thus  $(L_j^{(2)}:L_i^{(2)})\subseteq L_{j_{\ell}}$ and $(L_{k}^{(2)}:L_{i}^{(2)})=(L_{j_{\ell}}:L_{i_1})=L_{j_{\ell}}$.

(ii) If $\deg_{x_{j_{\ell}+1}}(L_j^{(2)})\leq\deg_{x_{j_{\ell}+1}}(L_i^{(2)})$ for all ${\ell}\in[2]$,  then  $L_{j_{\ell}+1}|^{edge}L_{i}^{(t)}$. Claim:  $j_{2}=n$.
In fact, if $j_{2}\ne n$, then $j_{2}+1\not\equiv 1\  (\text{mod\ }n)$, that is $L_{j_{2}+1}\neq L_{i_1}$
and $L_{j_{1}+1}\neq L_{i_1}$ since $2 \leq  j_1\leq j_{2}<n$, contradicting  the assumption  that $L_{i_1}|^{edge}L_{i}^{(2)}$. Therefore, $L_n|^{edge}L_{j}^{(2)}$ and $L_{i}^{(2)}=L_{j_{1}+1}L_{i_1}$.
 We also have two subcases:

(a) If  $j_{1}=n-2$, then  $L_{n-1}|^{edge}L_{i}^{(2)}$ and $L_{n-2}|^{edge}L_{j}^{(2)}$. In this case, we choose $L_k^{(2)}=L_{j}^{(2)}$, thus $(L_{k}^{(2)}:L_{i}^{(2)})=(x_{n-2})$ with  $L_n|^{edge}L_{k}^{(2)}$,  $L_{n-1}|^{edge}L_{i}^{(2)}$.

(b) If  $j_{1}\ne n-2$, then  $ j_{1}+1\ne n-1$, i.e., $L_{n-1}\nmid\!^{edge}L_{i}^{(2)}$. Note that $L_n\nmid\!^{edge}L_{i}^{(2)}$ since $j_2=n$,  $L_n|^{edge}L_{j}^{(2)}$ and the assumption that  $L_{j_\ell}\nmid^{edge}L_{i}^{(2)}$ for all $\ell\in [2]$. Therefore,
 $x_n\nmid\!L_{i}^{(t)}$, which forces $(L_j^{(2)}:L_i^{(2)})\subseteq (x_{n})$. In this case, we choose  $L_k^{(2)}=L_{n}\frac{L_{i}^{(2)}}{L_{i_1}}$, then $(L_{k}^{(2)}:L_{i}^{(2)})=(L_{n}:L_{i_1})=(x_{n})$.
\end{proof}

\begin{Theorem}\label{Li}
    Let $t$ be a positive integer and  let  $L_{i}^{(t)}, L_{j}^{(t)}\in L^{(t)}$ satisfying $L_1|^{edge}L_{i}^{(t)}$ and $L_{i}^{(t)}>L_{j}^{(t)}$. Then there exists some $L_{k}^{(t)}\in L^{(t)}$ with $L_{i}^{(t)}>L_{k}^{(t)}$ such that $(L_{j}^{(t)}:L_{i}^{(t)})\subseteq (L_{k}^{(t)}:L_{i}^{(t)})$ and $(L_{k}^{(t)}:L_{i}^{(t)})$ has one of the following two forms:
    \begin{itemize}
    \item[(1)] $(L_{k}^{(t)}:L_{i}^{(t)})=(L_{\ell_2}:L_{\ell_1})$ with $L_{\ell_1}>L_{\ell_2}$,  $L_{\ell_2}|^{edge}L_{k}^{(t)}$ and $L_{\ell_1}|^{edge}L_{i}^{(t)}$;
    \item[(2)] $(L_{k}^{(t)}:L_{i}^{(t)})=(x_{n-2d})$ for some $d$ with $L_{n-2s}|^{edge}L_{k}^{(t)}$ and $L_{n+1-2s}|^{edge}L_{i}^{(t)}$ for $s=0,1,\ldots, d$.
    \end{itemize}
\end{Theorem}
\begin{proof}
We prove the statement  by induction on $t$. The case where $t=1$ is trivial, and the case where $t=2$ follows from Lemma \ref{t=2}.
Now  we assume that $t\geq 3$. Let $L_{j}^{(t)}=L_{j_1}^{a_{j_1}}L_{j_2}^{a_{j_2}}\cdots L_{j_{k_j}}^{a_{j_{k_j}}}$ and
$L_{i}^{(t)}= L_{i_1}^{b_{i_1}}L_{i_2}^{b_{i_2}}\cdots L_{i_{k_i}}^{b_{i_{k_i}}}$, where  $1 \leq  j_1< \cdots < j_{k_{j}} \leq n$, $i_1=1$ and $2 \leq  i_2< \cdots < i_{k_{i}} \leq n$, and $\sum\limits_{\alpha=1}^{k_j}a_{j_{\alpha}}=\sum\limits_{\beta=1}^{k_i}b_{i_{\beta}}=t$  with  $a_{j_{\alpha}}$, $b_{i_{\beta}}>0$ for all $\alpha \in[k_j]$, $\beta \in[k_i]$.
 We consider the following two cases:

(I) If $L_{i_a}|^{edge}L_j^{(t)}$ for some $a\in[k_i]$. We also have two subcases:

(i) If $L_{i_1}|^{edge}\frac{L_{i}^{(t)}}{L_{i_a}}$ for some $a\in[k_i]$, then
$(L_j^{(t)}:L_i^{(t)})=(L_{j'}^{(t-1)}:L_{i'}^{(t-1)})$,
where $L_{j'}^{(t-1)}=\frac{L_{j}^{(t)}}{L_{i_a}}$ and $L_{i'}^{(t-1)}=\frac{L_{i}^{(t)}}{L_{i_a}}$.  By the induction hypothesis,  $(L_{j'}^{(t-1)}:L_{i'}^{(t-1)})\subseteq (L_{k'}^{(t-1)}:L_{i'}^{(t-1)})$ for some $L_{k'}^{(t-1)}\in L^{(t-1)}$ such that $L_{i'}^{(t-1)}>L_{k'}^{(t-1)}$
 and $(L_{k'}^{(t-1)}:L_{i'}^{(t-1)})$ has one of the following two forms:

(1) $(L_{k'}^{(t-1)}\!:\!L_{i'}^{(t-1)})=(L_{\ell_2}:L_{\ell_1})$ with  $L_{\ell_1}>L_{\ell_2}$,  $L_{\ell_2}|^{edge}L_{k'}^{(t-1)}$ and $L_{\ell_1}|^{edge}L_{i'}^{(t-1)}$.

(2) $(L_{k'}^{(t-1)}\!:\!L_{i'}^{(t-1)})=(x_{n-2d})$ for some   $d$ with $L_{n-2s}|^{edge}L_{k}^{(t)}$ and $L_{n+1-2s}|^{edge}L_{i}^{(t)}$ for $s=0,1,\ldots, d$.
In this case, we choose $L_k^{(t)}=L_{k'}^{(t-1)}L_{i_{a}}$, we get that  $L_{i}^{(t)}>L_{k}^{(t)}$ and $(L_{k}^{(t)}:L_{i}^{(t)})= (L_{k'}^{(t-1)}:L_{i'}^{(t-1)})$, as desired.

(ii) If for all  $L_{i_a}|^{edge}L_j^{(t)}$, we always have  $L_{i_1}\nmid\!^{edge}\frac{L_{i}^{(t)}}{L_{i_a}}$, then $j_{1}=1$ and $b_{j_1}=1$. Thus $a_{j_1}=1$ since  $L_{i}^{(t)}>L_{j}^{(t)}$.
Thus $L_{j}^{(t)}=L_{i_1}L_{j_2}^{a_{j_2}}\cdots L_{j_{k_j}}^{a_{j_{k_j}}}$,
$L_{i}^{(t)}= L_{i_1}L_{i_2}^{b_{i_2}}\cdots L_{i_{k_i}}^{b_{i_{k_i}}}$ with  $i_2,j_2\ge 2$. In this case, we can assume that $L_{j_\alpha}\nmid\!^{edge} \frac{L_{i}^{(t)}}{L_{i_1}}$ for all $2\le \alpha\le k_j$ and
$L_{i_\beta}\nmid\!^{edge} \frac{L_{j}^{(t)}}{L_{j_1}}$ for all $2\le \beta\le k_i$. Otherwise,  we can obtain the desired result by the induction hypothesis.

 Claim: $\deg_{x_{j_{\ell}+1}}(L_j^{(t)})>\deg_{x_{j_{\ell}+1}}(L_i^{(t)})$ for  some $2\leq\ell\leq k_j$.

 In fact, if  $\deg_{x_{j_\ell+1}}(L_j^{(t)})\leq\deg_{x_{j_\ell+1}}(L_i^{(t)})$ for all $2\leq\ell\leq k_j$, then  $L_{j_\ell+1}^{a_{j_\ell}}|^{edge}\frac{L_{i}^{(t)}}{L_{i_1}}$,
since  $\frac{L_{j}^{(t)}}{L_{i_1}}=L_{j_2}^{a_{j_2}}\cdots L_{j_{k_j}}^{a_{j_{k_j}}}$. Thus $\frac{L_{i}^{(t)}}{L_{i_1}}=L_{j_2+1}^{a_{j_2}}\cdots L_{j_{k_j}+1}^{a_{j_{k_j}}}$, since $\sum\limits_{\ell=2}^{k_j}a_{j_\ell}=t-1$. This contradicts the assumption that $L_{i}^{(t)}>L_{j}^{(t)}$.

  Therefore,  $L_{j_{\ell}'}|^{edge}L_{i}^{(t)}$ for some $j_{\ell}'\in\{j_{\ell}-1,j_{\ell}+1\}$, since $\supp(L_{j_{\ell}})\cap \supp(L_{i}^{(t)})\neq\emptyset$.
If  $L_{j_\ell-1}|^{edge}L_{i}^{(t)}$, then we choose $L_k^{(t)}=\frac{L_{i}^{(t)}}{L_{j_\ell-1}}L_{j_\ell}$. Thus   $(L_{k}^{(t)}:L_{i}^{(t)})=(L_{j_\ell}:L_{j_\ell-1})=(x_{j_\ell+1})$ and $(L_j^{(t)}:L_i^{(t)})\subseteq (x_{j_\ell+1})$ since $\deg_{x_{j_{\ell}+1}}(L_j^{(t)})>\deg_{x_{j_{\ell}+1}}(L_i^{(t)})$.
Otherwise, we have $L_{j_\ell-1}\!\nmid^{edge}\!L_{i}^{(t)}$ and $L_{j_\ell+1}|^{edge}L_{i}^{(t)}$. This forces $(L_j^{(t)}:L_i^{(t)})\subseteq L_{j_\ell}$. In this case, we  choose
\[
L_k^{(t)}=\left\{\begin{array}{ll}
		L_{j_\ell}\frac{L_{i}^{(t)}}{L_{i_2}}\ \ &\text{if $j_\ell=n$},\\
		L_{j_\ell}\frac{L_{i}^{(t)}}{L_{i_1}}\ \ &
		\text{else}.
	\end{array}\right.
\]
Thus  \[
(L_{k}^{(t)}:L_{i}^{(t)})=\left\{\begin{array}{ll}
		(L_{j_\ell}:L_{i_2})\ \ &\text{if $j_\ell=n$},\\
		(L_{j_\ell}:L_{i_1})\ \ &
		\text{else}.
	\end{array}\right.
\]

(II) If $L_{i_a}\!\nmid^{edge}\! L_j^{(t)}$ for any $a\in[k_i]$, then $j_1\ge 2$. The case where  $L_{j_\ell}|^{edge}L_{i}^{(t)}$ for some $\ell\in [t]$ follows from the condition that $L_{i_a}|^{edge}L_j^{(t)}$ for some $a\in [t]$. Now we assume that $L_{j_\ell}\nmid^{edge}L_{i}^{(t)}$ for all $\ell\in [t]$.
We have two subcases:

(i) If $\deg_{x_{j_{\ell}+1}}(L_j^{(t)})>\deg_{x_{j_{\ell}+1}}(L_i^{(t)})$ for  some $\ell\in[k_j]$, then the desired result holds by similar arguments as in the  case (ii) above.

(ii) If $\deg_{x_{j_\ell+1}}(L_j^{(t)})\leq\deg_{x_{j_\ell+1}}(L_i^{(t)})$ for all $\ell\in[k_j]$,  then  $L_{j_{\ell}+1}|^{edge}L_{i}^{(t)}$. Thus, we can write $L_{i}^{(t)}$ as $L_{i}^{(t)}= L_{j_1+1}^{a_{j_1}'}L_{j_2+1}^{a_{j_2}'}\cdots L_{j_{k_j+1}}^{a_{j_{k_j}}'}Q_i$ with $a_{j_\ell}'\geq a_{j_\ell}$ for all $\ell\in[k_j-1]$, where $Q_i$ is a monomial.  We claim that $j_{k_j}=n$. In fact, if $j_{k_j}\ne n$, then
$j_{k_j}+1\not\equiv 1\  (\text{mod\ }n)$, it follows that   $L_{j_{k_j}+1}\neq L_{i_1}$  and $a_{j_{k_j}}'\geq a_{j_{k_j}}$.
Note that $L_{j_{\ell}+1}\neq L_{i_1}$  for all $\ell\in[k_j]$, we have  $L_{i_1}|^{edge}Q_i$. In this case, we have  $\sum\limits_{\ell=1}^{k_j}a_{j_\ell}'+1\ge \sum\limits_{\ell=1}^{k_j}a_{j_\ell}+1=t+1$. This implies  $L_{i}^{(t)}\in \mathcal{G}(I(C_{\bf w}^n)^{t+1})$, a contradiction.
Therefore,  $L_n|^{edge}L_{j}^{(t)}$.
We consider the following two subcases:

(a) If $(L_j^{(t)}:L_i^{(t)})\subseteq (x_n)$, then we choose $L_k^{(t)}=L_{n}\frac{L_{i}^{(t)}}{L_{1}}$. Then   $(L_{k}^{(t)}:L_{i}^{(t)})=(L_{n}:L_{1})=(x_n)$.

(b) If $(L_j^{(t)}:L_i^{(t)})\not\subseteq (x_n)$, then  $x_{j_{k_j}}^{a_{j_{k_j}}}|L_i^{(t)}$  since $x_{j_{k_j}}^{a_{j_{k_j}}}|L_j^{(t)}$.  Note that $L_{j_{k_j}}|^{edge}\! L_j^{(t)}$ and $L_{i_a}\!\nmid^{edge}\! L_j^{(t)}$ for each $a\in[k_i]$, we have  $L_{j_{k_j}}\nmid^{edge}\! L_i^{(t)}$. Therefore, $L_{j_{k_j}-1}^{a_{j_{k_j}}}|^{edge}L_i^{(t)}$.

Claim:  $L_{j_{k_j}-2}|^{edge}L_j^{(t)}$. In fact, if $L_{j_{k_j}-2}\nmid^{edge}L_j^{(t)}$, we also write   $L_i^{(t)}$ as $L_i^{(t)}=L_{i_1}(\prod\limits_{\ell=1}^{k_j-1}L_{j_\ell+1}^{a'_{j_\ell}})L_{j_{k_j}-1}^{a_{j_{k_j}}}Q_i'$, where $Q_i'$ is a monomial.
By comparing the number of edges in the expression $L_i^{(t)}$, we get $t\geq\sum\limits_{\ell=1}^{k_j-1}a_{j_\ell}'+a_{j_{k_j}}+1\ge\sum\limits_{\ell=1}^{k_j}a_{j_\ell}+1=t+1$, a contradiction.
Therefore, $L_{j_{k_j}-2}|^{edge}L_j^{(t)}$. We can repeat the above discussion depending on  whether or not $(L_j^{(t)}:L_i^{(t)})$ is contained in $(x_{n-2d})$, where $d=0,1,\ldots$, This process can be performed only a finite number of times. And this concludes the proof.
\end{proof}

\begin{Theorem}\label{Ji}
	Let $t$ be a positive integer and    $L^{(t)}=\{L_1^{(t)}, L_2^{(t)}, \ldots, L_r^{(t)}\}$ be a totally ordered set of all elements of $\mathcal{G}(I(C_{{\bf w}}^n)^t)$ such that $L_1^{(t)}>\cdots>L_r^{(t)}$ and  $C=\{L_i^{(t)}: L_1|^{edge}L_i^{(t)} \text{\ for all\ }i\in[r]\}$. For each $i\in [c]$ with $c=|C|$, we write $L_i^{(t)}$ as $L_i^{(t)}=L_{i_1}^{a_{i_1}}\cdots L_{i_{k_i}}^{a_{i_{k_i}}}$ with  $1= i_1<\cdots < i_{k_i}\leq n$,
	$\sum\limits_{j=1}^{{k_i}}a_{i_j}=t$ and $a_{i_j}>0$ for each $1\leq j\leq k_i$. Suppose also that  $J_i=(L_{i+1}^{(t)}, \ldots, L_r^{(t)})$, $K_i=((L_2,\ldots,L_n)\!:\!L_1)+\sum\limits_{j=2}^{p_i}(L_{i_j+1}\!:\!L_{i_j})$ for each $i\in [c]$, where
	$p_i=\begin{cases}
		k_i-1,& \text{if $i_{k_i}=n$,}\\
		k_i,& \text{otherwise.}\\
	\end{cases}$ Then
$(J_i:L_{i}^{(t)})=K_i+Q_i$  for each $i\in [c]$, where  $Q_{i}=\sum\limits_{j=0}^{q_i}(x_{n-2j})$ and $q_i=\max\{\ell_i:L_{n+1-2s}|^{edge}L_{i}^{(t)}$ for all $0\leq s\leq \ell_i\}$.
\end{Theorem}

\begin{proof} If $t=1$, then $i=k_1=1$, since $L_1|^{edge}L_i$. The result is trivial. Now suppose $t\geq 2$. We can set $M_{b}=\frac{L_{i}^{(t)}}{L_{1}}L_{b}$ for each $2\leq b\leq n$, $N_{h}=\frac{L_{i}^{(t)}}{L_{i_{h}}}L_{i_{h}+1}$ for each $2\leq h\leq p_i$ and $T_{d}=\frac{L_{i}^{(t)}}{\prod\limits_{s=0}^{d}L_{n+1-2s}}\prod\limits_{s=0}^{d}L_{n-2s}$ for each $0\leq d\leq q_i$.  By the choice
of $p_i$ and $q_i$, we obtain  $\{M_{2},M_{3}\ldots,M_{n},N_{2},\ldots,N_{p_i},$ $ T_0,T_1,\ldots, T_{q_i}\}\subseteq J_{i}$. It follows that
\begin{eqnarray*}
	K_i+Q_i&=&((L_{2},\ldots, L_n):L_{1})+\sum\limits_{h=2}^{p_i}(L_{i_h+1}:L_{i_h}) +\sum\limits_{d=0}^{q_i}(x_{n-2d})\\
    &=&((M_{2},\ldots,M_{n}):L_{i}^{(t)})+((N_{2},\ldots,N_{p_i}):L_{i}^{(t)})
	+\sum\limits_{d=0}^{q_i}(T_{d}:L_{i}^{(t)})\\
	&=&((M_{2},\ldots,M_{n},N_{2},\ldots,N_{p_i},T_0,T_1,\ldots, T_{q_i}):L_{i}^{(t)})\subseteq (J_{i}:L_i^{(t)}).
\end{eqnarray*}
Note that $(J_{i}:L_{i}^{(t)})=((L_{i+1}^{(t)}, \ldots, L_r^{(t)}):L_{i}^{(t)})=\sum\limits_{j=i+1}^{r}(L_j^{(t)}:L_{i}^{(t)})$.  In order to prove $(J_{i}:L_i^{(t)})\subseteq K_i+Q_i$, it  suffices to prove $(L_j^{(t)}:L_{i}^{(t)})\subseteq  K_i+Q_i$ for each $i+1 \leq j \leq r$.  By Theorem
\ref{Li}, $(J_{i}:L_i^{(t)})$ has one of the following two cases:

(1) $(L_{j}^{(t)}\!:\!L_{i}^{(t)})\subseteq (L_{\ell_2}\!:\!L_{\ell_1})$, where $L_{\ell_1}>L_{\ell_2}$,  $L_{\ell_2}|^{edge}L_{j}^{(t)}$ and $L_{\ell_1}|^{edge}L_{i}^{(t)}$.  In this case, we have $\ell_{1}\leq i_{p_i}$. Note that if $\mbox{supp} (L_{\ell_1})\cap \mbox{supp} (L_{\ell_2})\neq\emptyset$,  then $\ell_{1}=\ell_2-1$ or $\ell_{1}=1$ and $ \ell_2=n$ since $L_{\ell_{1}}>L_{\ell_2}$.  Therefore, $(L_{\ell_2}:L_{\ell_1})\subseteq K_i$ by
\[
(L_j^{(t)}:L_{i}^{(t)})\subseteq (L_{\ell_2}:L_{\ell_1})=\left\{\begin{array}{ll}
	(L_{\ell_2}:L_{\ell_1}) \ &\text{if $\ell_{1}=\ell_2-1$,}\\
	(L_{n}:L_1) \ &\text{if $\ell_{1}=1$,~$\ell_2=n$,}\\
	(L_{\ell_2}) \ &\text{if $\mbox{supp} (L_{\ell_1})\cap \mbox{supp} (L_{\ell_2})=\emptyset$.}
\end{array}\right.
\]

(2)  $(L_j^{(t)}:L_{i}^{(t)})\subseteq (x_{n-2d})$ for some $d$ with  $L_{n-2s}|^{edge}L_{k}^{(t)}$ and  $L_{n+1-2s}|^{edge}L_{i}^{(t)}$, $s=0,1,\ldots, d$.  By the choice of $q_i$, we get $q_i \geq d$, which implies $(L_j^{(t)}:L_{i}^{(t)})\subseteq Q_i$.
\end{proof}

\begin{Theorem}
\label{n-cycle3} Under the assumptions of Setting \ref{1-edge-weighted},
we also assume that there is only one edge with non-trivial weights. Then, for all $t\ge 1$,
\[
\reg (S/I(C_{{\bf w}}^n)^t)=2\omega t+\lfloor \frac{n}{3} \rfloor-2
\]
where $\omega=\max\{{\bf w}_1,\ldots,{\bf w}_n\}$.
\end{Theorem}
\begin{proof}
Let $I=I(C_{{\bf w}}^n)$ and $L^{(t)}=\{L_1^{(t)}, L_2^{(t)}, \ldots, L_r^{(t)}\}$ be a totally ordered set of all elements of $\mathcal{G}(I(C_{{\bf w}}^n)^t)$. Let $C=\{L_i^{(t)}: L_1|^{edge}L_i^{(t)}\ \text{for\ }i\in[r]\}$.
 For each $i\in [c]$ with $c=|C|$, we write $L_i^{(t)}$ as $L_i^{(t)}=L_{i_1}^{a_{i_1}}\cdots L_{i_{k_i}}^{a_{i_{k_i}}}$ with  $1= i_1<\cdots < i_{k_i}\leq n$,
	$\sum\limits_{j=1}^{{k_i}}a_{i_j}=t$ and $a_{i_j}>0$ for each $1\leq j\leq k_i$.  Suppose also that  $J_i=(L_{i+1}^{(t)}, \ldots, L_r^{(t)})$ for each $i\in [c]$,  Consider the exact sequences:
\begin{gather*}
	\begin{matrix}
		0 & \longrightarrow & \frac{S}{(J_1:L_{1}^{(t)})}(-s_1)  & \stackrel{ \cdot L_{1}^{(t)}} \longrightarrow  & \frac{S}{J_1} & \longrightarrow & \frac{S}{I^t} & \longrightarrow & 0 \\
		0 & \longrightarrow & \frac{S}{(J_2:L_{2}^{(t)})}(-s_2)  & \stackrel{ \cdot L_{2}^{(t)}} \longrightarrow  & \frac{S}{J_2} & \longrightarrow & \frac{S}{J_1} & \longrightarrow & 0  \\
		0 & \longrightarrow & \frac{S}{(J_3:L_{3}^{(t)})}(-s_3) & \stackrel{ \cdot L_{3}^{(t)}} \longrightarrow & \frac{S}{J_3} &\longrightarrow & \frac{S}{J_2} & \longrightarrow & 0 \\
		&  &\vdots&  &\vdots&  &\vdots&  \\
		0&  \longrightarrow & \frac{S}{(J_{c}:L_{c}^{(t)})}(-s_c) & \stackrel{ \cdot L_{c-1}^{(t)}} \longrightarrow & \frac{S}{J_{c}}& \longrightarrow & \frac{S}{J_{c-1}}& \longrightarrow & 0
	\end{matrix}
\end{gather*}
where $s_i$ is the degree of $L_{i}^{(t)}$ for each $i\in [c]$. Thus $s_1=2w_1t$ and $2\omega+2(t-1)\leq s_i\leq 2\omega(t-1)+2$ for each $2\leq i\leq c$.
Note that  $J_c=I(P_{{\bf w}}^n)^t$ and $I(P_{{\bf w}}^n)=(x_2x_3,x_3x_4,\ldots,x_{n-1}x_n,x_nx_1)$. For each  $i\in [c]$, by Theorem \ref{Ji}, we have $(J_i:L_{i}^{(t)})=(x_3,x_n)+(x_4x_5,x_5x_6,\ldots,x_{n-2}x_{n-1})+\sum\limits_{j=2}^{p_i}(x_{i_j+2})+\sum\limits_{j=0}^{q_i}(x_{n-2j})$, where  $q_i=\max\{\ell_i: L_{n+1-2s}|^{edge}L_{i}^{(t)} \text{ for all } 0\leq s\leq \ell_i\}$ and $p_i=\begin{cases}
	k_i-1,& \text{if $i_{k_i}=n$,}\\
	k_i,& \text{otherwise.}
\end{cases}$ \\
Note that if  $i=1$, then $p_1=k_1=1$ and $q_1=0$. Thus $(J_1:L_{1}^{(t)})=(x_3,x_n)+(x_4x_5,x_5x_6,\ldots,x_{n-2}x_{n-1})$. By Lemma  \ref{path}, we get $\reg(S/(J_1:L_{1}^{(t)}))=\lfloor\frac{n}{3}\rfloor-1$  and
$\reg(S/J_c)=\lfloor\frac{n+1}{3}\rfloor+2(t-1)\leq2\omega t+\lfloor\frac{n}{3}\rfloor-2$ since $J_c=(x_2x_3,x_3x_4,\ldots,x_{n-1}x_n,x_nx_1)^t$. On the other hand, we can see from  Lemma  \ref{var} that
 $\reg(S/(J_i:L_{i}^{(t)}))\leq \reg(S/(J_1:L_{1}^{(t)}))=\lfloor\frac{n}{3}\rfloor-1$ for all $i\in [c]$.  Thus, the desired results follow from Lemma \ref{exact} and the above exact sequences.
\end{proof}

\medskip
\hspace{-6mm} {\bf Acknowledgments}

 \vspace{3mm}
\hspace{-6mm}  This research is supported by the Natural Science Foundation of Jiangsu Province (No. BK20221353). The authors are grateful to the software systems \cite{Co} and \cite{GS}
 for providing us with a large number of examples.

\end{document}